\documentclass[12pt]{elsarticle}
\usepackage[usenames]{color}
\usepackage{amssymb}
\usepackage{amsmath}
\usepackage{amsthm}
\usepackage{amsfonts}
\usepackage{amscd}
\usepackage{graphicx}
\usepackage[colorlinks=true,
linkcolor=webgreen,
filecolor=webbrown,
citecolor=webgreen]{hyperref}
\definecolor{webgreen}{rgb}{0,.5,0}
\definecolor{webbrown}{rgb}{.6,0,0}
\usepackage{color}
\usepackage{fullpage}
\usepackage{float}
\usepackage{graphics}
\usepackage{latexsym}
\usepackage{epsf}
\usepackage{bookmark}
\usepackage{mathrsfs}
\usepackage{hyperref}
\def\C{\mathbb{C}}
\def\Z{\mathbb{Z}}
\def\N{\mathbb{N}}
\def\cC{\mathcal C}
\def\cD{\mathcal D}
\def\diag{\hbox{diag}}
\def\hc{\hat c}

\newtheorem{prop}{Proposition}[section]
\newtheorem{cor}{Corollary}[section]

\begin{document}
\begin{frontmatter}
	\title{ Classification of the hypergeometric orthogonal polynomials via their recurrence coefficients}
\author{Luis Verde-Star}
 \address{
Department of Mathematics, Universidad Aut\'onoma Metropolitana, Iztapalapa, Mexico City,  Mexico}
\ead{verde@xanum.uam.mx}

\begin{abstract} 
We present a new  classification of the class $\cC$ of  all the hypergeometric orthogonal polynomial sequences.	The classification  uses  properties of the coefficients $\alpha_n$ and $\beta_n$ in the three-term recurrence relation satisfied by a  polynomial sequence $u_n(t)$. Such coefficients are rational functions of $n$ and are determined by a set of six parameters. In \cite{Uni} we obtained explicit expressions for the recurrence coefficients  for all the  hypergeometric and basic hypergeometric  orthogonal  polynomial sequences.
 
		The partial fraction decomposition of $\alpha_n$ is a  linear combination of five  linearly independent  functions of $n$ with coefficients $d_j$, for $0 \le j \le 4$, that are polynomials in  our six parameters.  For each value of a parameter $r$, associated with the eigenvalues, we classify the $\alpha_n$ according to the set of coefficients $d_j$ that are nonzero. There are certain particular values of $r$ that must be considered separately.

		Our classification yields a collection of 53 disjoint classes and it is different from the Askey scheme in several aspects. It is  similar to the classifications proposed by Koornwinder in \cite{KAsk, K2, K3}.


\end{abstract}
\end{frontmatter}

\section{Introduction}
The class of hypergeometric orthogonal polynomial sequences  has been exten\-siv\-ely studied for a long time, since the times of Chebyshev, Jacobi, and Legendre. See \cite{Chi} and \cite{Hyp}. An early classification of the class appears in \cite{Chi}, pages 217--221, where Chihara classifies the sequences according to the type of functions of $n$ of the recurrence coefficients. Until recently the Askey scheme \cite{Hyp} was considered as the best description of the class and the relations among its members. In the recent papers \cite{KAsk, K2, K3} Koornwinder pointed out several deficiencies of the Askey scheme and proposed new ways to describe the class and its properties. Some of his constructions are based on our results in \cite{Uni}.

In \cite{Uni} we presented a unified construction of  all the hypergeometric orthogonal and  $q$-orthogonal polynomial sequences and proved that they sat\-isfy a gen\-er\-al\-ized dif\-fer\-ence-eigenvalue equation of order one, with respect to a Newtonian basis for the space ${\mathbb C}[t]$. We also obtained a uniform parametrization using 7 parameters which are the coefficients of three quadratic polynomials in $n$ that define the nodes of the Newtonian basis, the eigenvalues, and the generalized difference operator. The number of parameters can be reduced to 4, but that would complicate most of the development of the theory. 

The traditional way to express the recurrence coefficients $\alpha_n$ and $\beta_n$ is as a quotient of polynomials with the numerator written as a product of several factors of the form $n+c_k$, where the $c_k$ depend on some parameters. This is how it is done in most cases in \cite{Hyp}, and it is the source of numerous problems for the classification of the polynomial sequences and also for the description of the connections among them. From the  Wilson  and the Askey-Wilson polynomials,  which are considered as the most general in the Askey and the $q$-Askey schemes, it is not possible to obtain all the other sequences by giving appropriate values to the parameters.  As a consequence of this fact some limit processes are required to describe how to pass from one family to another one. We will see that the Wilson polynomials are particular elements of the class whose  $\alpha_n$ has five nonzero coefficients in its partial fraction decomposition, but for the Wilson polynomials one of them is always equal to 1, independent on the parameters. The  same situation occurs with the Askey-Wilson polynomials in the $q$-Askey scheme.

In the present paper we deal only with the class $\cC$ of hypergeometric orthogonal polynomial sequences. This class contains all the sequences in the Askey scheme \cite{Hyp}. 

Our main objective is to construct a classification of the class $\cC$ based on properties of the coefficients in the three-term recurrence relation 
		$u_{n+1}(t)=(t-\beta_n) u_n(t) - \alpha_n u_{n-1}(t)$. 
 For this purpose we obtain first several new properties of the recurrence coefficients, starting from general results obtained in \cite{Uni}.
 The new ingredient in this paper is the use of the partial fractions decomposition of the coefficients $\alpha_n$. The decomposition  is expressed as a linear combination  of linearly independent rational functions of $n$ and a parameter $r$ related with the sequence of eigenvalues. We also found that when $r$ is nonzero the sequence $\alpha_n$ is invariant under the change of variables $n \rightarrow -n + 2 -1/r$, which is an involution on the space of rational functions. This fact explains why the roots and the poles of $\alpha_n$ must appear by pairs and why the traditional way to express $\alpha_n$ is not sufficiently flexible.

 The parameter $r$ plays a central role in the classification. If $r$ is not in $\{0,1,1/2, 1/3\}$ the set of all the $\alpha_n$ associated with $r$, denoted by $\cC_r$, is divided into 16 nonempty disjoint classes. Since the poles of $\alpha_n$ are determined by $r$, it is easy to see that if $r \ne {\tilde r}$ then $\cC_r$ and $\cC_{\tilde r}$ are disjoint. 
For $r=0$  we found that $\cC_0$ is the set of all the polynomials   in $n$ of degree 4, 2, or 1 
 that have $n$ as a factor. Several polynomial sequences in the Askey scheme are associated with this set. 
 Each of $\cC_1$, $\cC_{1/2}$, and $\cC_{1/3}$ is divided into 12 nonempty and disjoint classes.

For every value of $r$ and every class in $\cC_r$  we can find explicit formulas for the general form of $\alpha_n$ in the given class in terms of our parameters. This allows us to construct concrete examples in every class.

 The approach used in this paper can also be used to obtain a classification of the $q$-orthogonal polynomials, for $q\ne 1.$

The paper is organized as follows.
In Section 2 we present a brief account of the construction of the bispectral $q$-hypergeometric orthogonal polynomials that we obtained in \cite{Uni} using our algebraic approach introduced in \cite{Mops} and \cite{Rec}. In Section 3 we consider the case with $q=1$ and obtain some new  properties of the recurrence coefficients. 
In Section 4 we present the partial fractions decomposition of the most general $\alpha_n$.
The classification defined by means of the partial fractions decomposition is presented in Section 5.
 In Section 6 we deal with the  classification of the sets $\cC_r$  with $r$ in $\{1, 1/2, 1/3\}$.  And finally, in Section 7, for each of the families of orthogonal polynomial sequences listed in the Askey scheme  we determine the class to which it is associated. We found that they are associated with only six classes.

\section{Construction of  bispectral hypergeometric orthogonal polynomials}

In this section we present a  brief account of the construction of the 
 bispectral hypergeometric orthogonal polynomials presented in our previous paper \cite{Uni}. 

Consider the linear  difference equation
\begin{equation}\label{eq:diffeq}
 s_{k+3} = z ( s_{k+2} -s_{k+1}) + s_k, \qquad k\ge 0,
\end{equation}
where $z$ is a nonzero complex number and $s_k$ is a sequence of complex numbers with initial terms $s_0,s_1,s_2$. Since the characteristic polynomial is $t^3 - z t^2 + z t -1$ we see that the product of the roots is equal to one, the sum of the roots is equal to $z$, and 1 is a root. Therefore 
the characteristic roots of the difference equation are $1$, $q$, and $q^{-1}$,  where $1 + q + q^{-1}=z$. If $z=3$ then $q=1$ is a triple root. If $z=-1$ then $q=-1$ is a double root, and if $z\ne 3$ and $z\ne -1$ then the roots $1, q, q^{-1}$ are distinct. Therefore the  general solution of \eqref{eq:diffeq} is of the form 
$s_k= d_0 + d_1 q^k + d_2 q^{-k}$ when $z\ne 3$ and $z \ne -1$ and becomes $s_k=d_0 + d_1 k + d_2 k^2$ when $z=3$, and $s_k= d_0 +d_1 (-1)^k + d_2 k ( -1)^k$ when $z=-1$. These lattices were considered by Bochner and Hahn in their classification of the classical and the $q$-orthogonal polynomials. Such lattices also appear in \cite{NSU}, where they are obtained from the solutions of  a second order difference equation \cite[eq. 3.1.12]{NSU} that must satisfy certain additional conditions.  

Let $x_k$, $h_k$, and $e_k$ be 3  solutions of \eqref{eq:diffeq}. These sequences will be used to construct a basis for the space  ${\mathbb C}[t]$, a linear operator $\cD$ on the space of polynomials, and a sequence of orthogonal polynomials $u_k(t)$ that are eigenfunctions of $\cD$ with eigenvalues $h_k$. The sequence $e_k$ is used to simplify the definition of the operator $\cD$ and provides additional parameters that are needed to construct a theory that can be applied to all the hypergeometric and basic hypergeometric orthogonal polynomials in the Askey schemes. 

The sequence $x_k$ determines the Newtonian basis $\{v_n(t): n\ge 0\}$ of the complex vector space ${\mathbb C}[t]$ of polynomials in $t$, defined by 
\begin{equation}\label{eq:Newton}
	v_n(t)= (t-x_0)(t-x_1)\cdots (t-x_{n-1}), \qquad n \ge 1,
\end{equation}
and $v_0(t)=1.$

We define the sequence $g_k$ by
\begin{equation} \label{eq:g}
g_k= x_{k-1} (h_k- h_0) + e_k, \qquad k \ge 1, 
\end{equation}
and $g_0=0$. 
This sequence satisfies a linear difference equation of order five. We add  the sequence $e_k$ to avoid some complicated restrictions on the initial values of $g_k$.  Let us suppose that $h_k \ne h_j$ if $k \ne j$, and $g_k \ne 0$ for $k \ge 1$. 

We use the basis $\{v_k: k \ge 0\}$ to define the linear operator $\cD$   by
\begin{equation}\label{eq:operD}
\cD v_k = h_k v_k + g_k v_{k-1}, \qquad k \ge 1.
\end{equation}
	Since $g_0=0 $ we see that $\cD t^n$ is equal to $ h_n t^n $ plus a  polynomial of lower degree. 
The operator $\cD$ is a generalized difference operator. 

	For $n \ge 0$ we define $u_n$ as the monic polynomial of degree $n$ which is an eigenfunction of $\cD$ with eigenvalue $h_n$. That is 
\begin{equation}\label{eq:eigenEq}
\cD u_n = h_n u_n, \qquad n \ge 0.
\end{equation}

 In \cite[p. 249]{Uni} we showed that
\begin{equation}\label{eq:u}
 u_n(t) = \sum_{k=0}^n c_{n,k} v_k(t), \qquad n \ge 0,
\end{equation}
 where the coefficients $ c_{n,k}$ are given by
 \begin{equation}\label{eq:cnk}
c_{n,k}=\prod_{j=k}^{n-1} \frac{g_{j+1}}{h_n -h_j}, \qquad 0 \le k \le n-1,
 \end{equation}
and $c_{n,n}=1$ for $n \ge 0$.
This expression for $u_n(t)$ was also obtained in \cite{VZ} using a different approach. The idea of representing orthogonal polynomials in terms of a Newtonian basis was introduced by Geronimus in \cite{Ger}.

The matrix $C=[c_{n,k}]$, where the coefficients $c_{n,k}$ are defined in \eqref{eq:cnk}, is lower triangular and all its entries in the main diagonal are equal to $1$. Therefore $C$ is invertible.     Let  $ C^{-1}=[{\hat c}_{n,k}]$. Using some properties of divided differences we proved in  \cite[p. 251]{Uni} that
\begin{equation}\label{eq:Cinv}
      \hc_{n,k}=\prod_{j=k+1}^n \frac{g_j}{h_k - h_j}, \qquad 0 \le k \le n-1,
\end{equation}
         and $\hc_{n,n}=1$ for $n \ge 0.$

 The entries in the 0-th column of $C^{-1}$ are
	\begin{equation}\label{eq:moments}
 \hc_{n,0} =\prod_{k=1}^n \frac{g_k}{h_0-h_k}, \qquad n \ge 1, 
	\end{equation}
and $\hc_{0,0}=1$. We denote them by $m_n=\hc_{n,0}$ for $ n \ge 0$.

In \cite{Uni} we also proved that the monic  polynomial sequence $u_n(t)$ satisfies a three-term recurrence relation of the form
\begin{equation}\label{eq:3term}
	u_{n+1}(t)= (t-\beta_n) u_n(t) - \alpha_n u_{n-1}(t), \qquad n \ge 1, 
\end{equation}
where the coefficients are given by
\begin{equation}\label{eq:beta}
	\beta_n= x_n + \frac{g_{n+1}}{h_n - h_{n+1}} -\frac{g_n}{h_{n-1} - h_n}, \qquad n\ge 0, 
\end{equation}
and 
\begin{equation}\label{eq:alpha}
 \alpha_n = \frac{g_n}{h_{n-1} -h_n} \left(\frac{g_{n-1}}{h_{n-2} - h_n} - \frac{g_n}{h_{n-1} - h_n} + \frac{g_{n+1}}{h_{n-1}-h_{n+1}} +x_n - x_{n-1} \right), \qquad n\ge 1. 
\end{equation}
Since $g_0=0$ the terms in the previous equations where $h_{-1}$ appears are equal to zero. 

If all the $\alpha_n$ are positive and the $\beta_n$ are real then the sequence $u_n$ is orthogonal with respect to a positive measure, and if all the $\alpha_n$ are nonzero then $u_n$ is orthogonal with respect to a not necessarily positive definite moments functional.

The numbers $m_n$  are the generalized moments of the polynomial sequence $u_k(t)$ with respect to the Newton basis $\{ v_k(t): k \ge 0 \}$. 
From \eqref{eq:moments} we see that  $m_n$ satisfies the  recurrence relation 
$$ m_{n+1}=\dfrac{g_n}{h_0 -h_n} m_n, \qquad n\ge 1. $$

When $z=3$ the characteristic polynomial of the difference equation  \eqref{eq:diffeq} has 1 as  a triple root and therefore every solution of \eqref{eq:diffeq} can be expressed in the form 
\begin{equation} \label{eq:gensol}
	s_k= d_0 + d_1 k + d_2 k (k-1),
\end{equation}
where $d_0, d_1,d_2$ are constants.
Therefore we can write the sequences $x_k$, $h_k$, and $e_k$ as follows
\begin{eqnarray}\label{eq:xhe}
	x_k& =& b_0 + b_1 k + b_2 k (k-1),  \cr
	h_k& =& a_0  + a_1 k + a_2 k (k-1),  \cr
	e_k& =& f_0 + f_1 k + f_2 k (k-1).  
\end{eqnarray}
Since $g_0= e_0=0$ we see that $f_0$ must be zero. The constant $a_0$ can be taken as zero because the eigenvalues $h_k$  appear only in terms of the form  $h_n - h_m$ in all the relevant formulas.

From \eqref{eq:xhe} it is clear that the sequences $x_k, h_k, e_k$ are well defined for all $k$ in $\Z$.

The properties of the polynomial sequence represented by the matrix $C$ can be expressed in terms of matrix equations as we describe next. 
Let $X$ be the shift matrix defined by $X_{k,k+1}=1$ for $k\ge 0$ and all the other entries equal zero.
 The matrix $S$ is the transpose of $X$. The diagonal matrix $H=\diag(h_0,h_1,h_2,\ldots)$ is called the matrix of eigenvalues. The diagonal matrix $F=\diag(x_0,x_1,x_2,\ldots)$ is the matrix of nodes associated with the Newtonian basis. The diagonal matrix $G=\diag(g_1,g_2,g_3,\ldots)$ is associated with the operator $\cD$. The Jacobi matrix $L$ is defined by $L_{k,k+1}=1,\ L_{k,k}=\beta_k,\ L_{k+1,k}=\alpha_{k+1}$, for $k\ge 0$.

 With respect to the Newtonian basis the operator of multiplication by the independent variable $t$ is represented by $ X+F$, and the operator $\cD$ by $( H + S G)$. 
 The three-term recurrence relation \eqref{eq:3term} corresponds to the equation
 \begin{equation}\label{eq:LCXF}
	 LC=C(X+F)
 \end{equation}
and the difference-eigenvalue equation \eqref{eq:eigenEq} is expressed  in terms of matrices as
\begin{equation}\label{eq:HC}
	C(H+SG)=HC.
\end{equation}
See \cite{Uni} for a more detailed account of these matrix equations. 

\section{Some properties of the recurrence coefficients}

We consider now the sequences of orthogonal polynomials obtained when the parameter $z$ in the difference equation \eqref{eq:diffeq} is equal to $3$  and therefore $q=1$. These are the sequences in the Askey scheme. See \cite{Hyp}.

Let us note that several equations in the previous section contain terms of the form $ g_k/(h_n- h_{n-m})$.
In order to simplify those equations it is convenient to define 
\begin{equation}\label{eq:gamma}
	\gamma(n,m)=\dfrac{ m \, g_{n+1-m}}{h_n-h_{n-m}}, \qquad n,m \in \N.
\end{equation}

From \eqref{eq:cnk} we obtain
\begin{equation}\label{eq:cnkgamma}
c_{n,k}=\prod_{j=1}^{n-k} \dfrac{\gamma(n,j)}{j},
\end{equation}
where $c_{n,k}$ are the coefficients of the orthogonal polynomials $u_n(t)$, defined in \eqref{eq:u}.

Using  equations   \eqref{eq:g}  and \eqref{eq:xhe} we obtain
\begin{equation}\label{eq:gammaexpl}
	\gamma(n,m)=\dfrac{x_{n-m}  (n-m+1) (a_1 + (n-m) a_2)+(n-m+1) ( f_1+ (n-m) f_2)}{a_1+ ( 2 n -m -1) a_2 }.
\end{equation}

It is easy to verify that $\gamma(n,m)$ is invariant under the rescaling
\begin{equation}\label{eq:rescaling}
 (a_1, a_2, f_1, f_2) \rightarrow t (a_1, a_2, f_1, f_2), \qquad t\in \C. 
\end{equation}

Now we can write \eqref{eq:beta} and \eqref{eq:alpha} as
\begin{equation}\label{eq:betagam}
	\beta_n=x_n -(\gamma(n+1,1)-\gamma(n,1)),
\end{equation}
and
\begin{equation}\label{eq:alphagam}
	\alpha_n= \gamma(n,1) \left( \dfrac{ \gamma(n,0) -2 \gamma(n,1) + \gamma(n,2)}{2} -(x_n-x_{n-1} ) \right).
\end{equation}
Let us note that these equations show that $\beta_n$ and $\alpha_n$ are invariant under the rescaling defined in \eqref{eq:rescaling}.

A simple computation shows that if $a_1=0$ then $(n-1)^2$ is a factor of the denominator of $\alpha_n$ and therefore $\alpha_1$ is not defined. The factor $(n-1)^2$ in the denominator 
 can be cancelled only when $f_1=0$, but if $a_1=f_1=0$ then  $n-2$ is a factor of the numerator of $\alpha_n$ which can not be cancelled, since in this case the roots of the denominator are $1/2$ and $1/3$.  Therefore, if we want to have $\alpha_n\neq 0$ for at least $n=1,2,3$ then $a_1$ must be nonzero. This fact and the invariance under the rescaling  \eqref{eq:rescaling} allow us to introduce a change of parameters that simplifies most of our equations and reduces the number of parameters.

 We use the substitutions
 \begin{equation}\label{eq:subsr}
  a_2= r a_1, \quad f_1= s_1 a_1, \quad f_2= s_2 a_1,
\end{equation}
and obtain
\begin{equation}\label{eq:gammar}
	\gamma(n,m,r)=\dfrac{(n-m+1) ((nr-m r  +1) x_{n-m} + (n-m) s_2 + s_1)}  { (2 n -m -1) r +1},
\end{equation}
\begin{equation}\label{eq:betar}
	\beta(n,r)= x_n -(\gamma(n+1,1,r) - \gamma(n,1,r)),
\end{equation}
and
\begin{equation}\label{eq:alphar}
	\alpha(n,r)= \gamma(n,1,r) \left( \dfrac{\gamma(n,0,r) - 2 \gamma(n,1,r) + \gamma(n,2,r)}{2}  -( x_n- x_{n-1} )\right).
\end{equation}
Note that $ a_1$ and $a_2$ do not appear in these equations.
The functions $\gamma, \alpha,$ and $\beta$ are also functions of the vector of  parameters $v=( b_0, b_1, b_2, s_1, s_2)$, but in order to simplify the notation we only  write $n$ and $r$ as arguments of $\alpha$ and $\beta$, and $n, m$ and $r$ as arguments of $\gamma$. We will consider later some substitutions in the vector of parameters.  

Since we want $\alpha(n,r)$ and $\beta(n,r)$ to be defined for  $n\ge 0$, and $r$ appears in their denominators, it is easy to verify that the denominators are nonzero if $r$ is not of the form $-1/m$, where $m $ is a positive integer. From now on we will suppose that $r$ satisfies such condition.  We will see that the values $0, 1, 1/2,$ and $1/3$ of $r$ correspond to recurrence coefficients with especial properties. 

From \eqref{eq:betar} and \eqref{eq:alphar} we obtain
\begin{eqnarray}\label{ eq:initial}
	\beta(0,r)&=&-s_1, \cr
	\alpha(1,r)&=& \dfrac{(s_1+ b_0) ( s_2 - r s_1)}{r+1}.
\end{eqnarray}
Since $ \alpha(1,r)$ must be nonzero we see that $s_1+ b_0$  and $ s_2 - r s_1$ must be nonzero.

Let $N$ denote the operator of substitution of parameters that sends the vector  $v=(b_0, b_1, b_2, s_1, s_2)$  to $-v$. 
It is easy to verify that $\alpha(n,r)$ is invariant under $N$ and also that if we apply $N $ to $\beta(n,r)$ we obtain $-\beta(n,r)$. 

We consider next the case with $r=0$, obtained when $a_2=0$ and hence the eigenvalues are of the form  $h_k= a_1 k$,  for $k\ge 0$, and $a_1 \ne 0$.

Substitution of $r=0$ in \eqref{eq:alphar} and \eqref{eq:betar} and a simple computation gives
\begin{equation}\label{eq:alrzero}
	\alpha(n,0)= n ( (n-1) b_2 +s_2) ( (n^2-3 n +2) b_2 +(n-1) ( b_1+ s_2) +b_0 +s_1),
\end{equation}
and
\begin{equation}\label{eq:berzero}
	\beta(n,0)= -( 2 n (n-1) b_2 +n ( b_1+ 2 s_2) + s_1).
\end{equation}
It is easy to see that $\alpha(n,0)$ is a polynomial that has a root at zero and  whose degree can be 4, 2, or 1, and $\beta(n,0)$ can be identically zero or  a polynomial whose degree can be 2, 1, or zero.

{\bf Definition.} We denote by ${\mathcal C}_0$ the set of all the polynomials $\alpha(n,0)$ of the form given in equation \eqref{eq:alrzero}.

If $b_2$ is nonzero then $\alpha(n,0)$ has degree 4, and for any three given numbers $t_1,t_2,t_3$ we can find values of $b_1, b_0, s_1, s_2$ for which the roots of  $\alpha(n,0)$  are $0, t_1,t_2, t_3$.

If $b_2=0$ and $s_2$ and $b_1+ s_2$ are nonzero then $\alpha(n,0)$ has degree 2. In this case for any given number $t$ we can find values for $s_1$ and $b_0$ for which zero and $t$ are roots of   $\alpha(n,0)$. If $b_2=0$ and $s_2=-b_1$ then $\alpha(n,0)=-b_1(b_0+s_1) n$. Therefore ${\mathcal C}_0$ is the set of all the nonzero polynomials in $n$ whose degree is 4, 2, or 1 and  have a root at $n=0$.

We consider next the case with  $r\neq 0$. In this case the eigenvalues $h_k$ are given by a quadratic function of $k$.
Let  $r\ne 0$. We  define  an involution $R$ on the set of rational functions of a complex variable $z$ as follows
\begin{equation}\label{eq:invol}
	R( w(z))= w\left(-z + 2 -\frac{1}{r}\right), 
\end{equation}
where $w(z)$ is a rational function of $z$ with some parameters.

Note that $R$ is a linear and multiplicative operator on the set of  rational functions of $z$,  $R(R(w(z)))=w(z)$ for every rational function $w(z)$, and if $p(z)$ is a polynomial then $R(p(z))$ is a polynomial of the same degree as $p(z)$. 

Since $\alpha(n,r)$ and $\beta(n,r)$ are rational functions of $n$ we can apply to them the involution $R$ with respect to the variable $n$.
A straightforward computation gives us the following proposition.

\begin{prop}\label{RalRbe}
	\begin{eqnarray} \label{eq:prop31}
		R(\alpha(n,r))& =&\alpha(n,r), \cr
		R(\beta(n,r)) &=& \beta(n-1,r), \cr
		\alpha(n,r) &=& \dfrac{-r^2 \gamma(n,1,r) R(\gamma(n,1,r))}{(2nr -r+1)( 2 n r -3 r +1)}. 
	\end{eqnarray}
\end{prop}

From  the third equation in the previous proposition we can obtain factorizations of $\alpha(n,r)$. We obtain next a factorization of  $\alpha(n,r)$  that is related with $\beta(n,r)$ in a simple way.

Define
\begin{equation}\label{eq:lambda}
	\lambda(n,r)=\dfrac{(n r - 2 r +1) \gamma(n,1,r)}{n (2 n r -3 r +1)},
\end{equation}
and
\begin{equation}\label{eq:mu}
\mu(n,r)= R(\lambda(n,r)).
\end{equation}
Since $R$ is an involution we also have $\lambda(n,r)= R(\mu(n,r))$.

These definitions combined with \eqref{eq:gammar} give us 
\begin{equation}\label{eq:lambdaexpl}
	\lambda(n,r)=\dfrac{(rn -2r+1) ( (r n-r+1 ) x_{n-1}  + s_2 (n-1) + s1)}{(2rn -2r+1) ( 2rn -3r+1)},
\end{equation}
and
\begin{equation}\label{eq:muexpl}
	\mu(n,r)=\dfrac{n ( (n-1) ( x_{-n+1} r^2 + (2 r n -r +1) b_2 - r( b_1-s_2) ) - r s_1 + s_2) }{(2nr-r+1) ( 2nr-2r +1)}.
\end{equation}

\begin{prop}\label{Rfactor} 
	\begin{eqnarray*}
		\alpha(n,r)&=& \lambda(n,r) \mu(n,r),\qquad n \ge 1, \cr
		\beta(n,r)&=& x_0 -\lambda(n+1,r) -\mu(n,r),\qquad n\ge 0.
	\end{eqnarray*}
\end{prop}
The proof follows from some simple computations.  Let us recall that $x_0= b_0$.
	This proposition is valid for any acceptable $r$, including  $r=0$,   and holds for  all the polynomial sequences in the Askey scheme. 

A factorization of $\alpha_n$ analogous to the one in Proposition \ref{Rfactor} appears in  \cite{Hyp}, where two sequences $A_n, C_n$  are  used to define the recurrence coefficients of some of the polynomials in the Askey scheme. For example, for the Wilson polynomials in equation (9.1.5), for the continuous Hahn polynomials in equation (9.4.4), and for the Hahn polynomials in equation (9.5.4).

 In \cite{Hyp} the numerator of the coefficients $\alpha_n$ for all the sequences in the Askey scheme is  expressed  as a product of factors of the form $n+t$, where $t$ is a constant. The most general $\alpha(n,r)$ can be written in the factorized form  
 \begin{equation}\label{eq:factorized}
	 \alpha(n,r)= \dfrac{n ((n-2)r +1)  \prod_{k=1}^3 ( c_k n -1) ( ( c_k n -2 c_k +1) r+ c_k)}{(2rn -r +1) ( 2 rn -2r+1)^2 ( 2 rn - 3 r +1)},
\end{equation}
where the  $c_k$, for $1 \le k \le 3$,  are constants and they are the reciprocals of some of the roots of $\alpha$. Note that here we have four parameters. This expression for $\alpha$ is more convenient than the ones in \cite{Hyp}. 

It is easy to transform \eqref{eq:factorized} into  the third equation in \eqref{eq:prop31}. For $r$ fixed we can find $b_2, b_1, b_0, s_2$ in terms of $c_1,c_2,c_3$ and $s_1$, where $s_1$ is arbitrary.  The reverse transformation requires the computation of the roots of $\gamma(n,1,r)$. 

\section{Partial fractions decomposition of $\alpha(n,r)$}

In this section we obtain a partial fractions decomposition of the rational function $\alpha(n,r)$, when  $r$ is nonzero.   We will use the coefficients in the decomposition  to  classify  the sequences $\alpha(n,r)$. The basic properties of partial fractions decompositions are presented in \cite{Rat}.

From \eqref{eq:alphar} or from  \eqref{eq:prop31}  we obtain
\begin{equation}\label{eq:alpharexpl}
	\alpha(n,r)= \dfrac{ n  ( n r - 2 r +1)\, (-r R(w)  ) \, w  }{(2 n r -r +1)   (2 n r  -2r +1)^2 (2 n r -3 r +1)},
\end{equation}
where
\begin{equation*}
	w= (nr -r +1) x_{n-1} +(n-1)s_2 + s_1.
\end{equation*}
Note that $w$ is a polynomial in $n$ of degree 3 and so is $R(w)$. Therefore,
 as a function of $n$,  $\alpha(n,r)$ is a rational function and its  numerator has degree 8 and its denominator has degree 4. It has simple poles at $n= (r-1)/2r$ and $n=(3 r -1)/2r$ and a double pole at $n=(2r-1)/2r$. Therefore, in the general case, the partial fractions decomposition of $\alpha(n,r)$ is a polynomial of degree 4 plus a linear combination of 3 linearly independent proper rational functions whose poles depend on $r$.

We define the functions $y_k$  of $n$ and $r$ as follows
\begin{equation}\label{eq:ypoly}
	y_0=(4 r^2 n)^2  (n r-2 r +1)^2, \qquad y_1=4 r^3 n (n r -2 r+1), \qquad y_2=1,
\end{equation}

\begin{equation}\label{eq:raty}
	y_3=\dfrac{(r-1)(3 r-1)}{(2n r- r  +1)(2n r-3 r+1)}, \qquad y_4=\dfrac{-(2r-1)^2}{(2nr-2r+1)^2},
\end{equation}
and the polynomials  
\begin{eqnarray}\label{eq:auxpoly}
	p_1&=&(b_2^2+ ( 8 b_2(b_0+b_1)- 4 b_1^2) r^2+ 4 b_2 ( 2 s_2 + b_1 -2 b_2)r - 2 b_2^2,\cr
	p_2&=& (4 b_0 +2 b_1 -b_2) r^3 + (8 s_1 + 4 s_2 + 4 b_0 - b_2) r^2 +(b_2-4 s_2 - 2 b_1) r + b_2,\cr
	p_3&=&(2 b_1-4 b_0 -3 b_2)r^3+ (8 s_1 -4 s_2 + 4 b_0 - b_2) r^2+ 3 b_2- 4 s_2 - 2 b_1) r + b_2,\cr
	p_4&=&(8s_1+4 b_0) r^2+( 2 b_2-4 s_2 - 2 b_1) r + b_2.
\end{eqnarray}

Now define
\begin{equation}\label{eq:coefd}
d_0=b_2^2, \qquad d_1=p_1, \qquad d_3= p_2\, p_3,\qquad  d_4=p_4^2, \qquad d_2=d_4- d_3.
\end{equation}

{\bf Remark.} The functions $y_k$, for $0 \le k \le 4$, are invariant under the involution $R$ and the polynomials $ d_k $, for $0 \le k \le 4$,  are invariant under the change of parameters $N$.

Now it is easy to verify that
\begin{equation}\label{eq:pfd}
	\alpha(n,r)=\dfrac{1}{256 r^6} ( d_0 y_0 + d_1 y_1+ d_2 y_2 + d_3 y_3 + d_4 y_4).
\end{equation}
This is the partial fractions decomposition of $\alpha(n,r)$. Note that the rational functions that correspond to the two simple poles are grouped in $y_3$. From \eqref{eq:raty} we notice that $y_3$ becomes zero when $r=1$ or $r=1/3$, and $y_4$ becomes zero when $r=1/2$. We will deal with these especial cases later. In this section we suppose that $r$ is not in $\{0,1,1/2,1/3 \}$ and it is not the reciprocal of a negative integer. Let us note that $\alpha(n,r)$ is always invariant under $R$, even when some of the coefficients $d_k$ are zero.

For a fixed $r$ every vector of the form  $(d_0, d_1, d_4-d_3, d_3, d_4)$ determines $\alpha(n,r)$ by \eqref{eq:pfd}.
The coefficients $d_k$, for $0 \le k \le 4$, can also be expressed in terms of $r$ and  the numbers $c_k$  of \eqref{eq:factorized}.

The sequence $\beta(n,r)$ can be written as 
\begin{equation}\label{eq:betapfd}
	\beta(n,r)= \dfrac{n (nr -r +1)}{2 r} \left(\dfrac{p_4}{(2 r n -2 r +1)(2rn+1)} - b_2\right) - s_1.
\end{equation}
This equation shows that, given $r \ne 0$, the function  $\beta(n,r)$ is determined by $p_4, b_2$, and  $s_1$. Let us  recall that $p_4^2=d_4$, $b_2^2= d_0$, and $- s_1=\beta(0,r)$ for all $r$.

Given a nonzero number $r$ and a vector  $v=(b_0,b_1,b_2, s_1, s_2)$ equation \eqref{eq:alpharexpl} determines a unique function $\alpha(n,r)$, but the same function can also be obtained with other vectors of parameters and the same $r$. For example, with $-v$, since $\alpha(n,r)$ is invariant under the substitution $N$.  

We define next  a modified  $\beta(n,r)$ as follows.
\begin{equation}\label{eq:tilbetapfd}
	\tilde\beta(n,r)= \dfrac{n (nr -r +1)}{2 r} \left(\dfrac{p_4}{(2 r n -2 r +1)(2rn+1)} + b_2\right) - s_1.
\end{equation}
We also have
\begin{equation}\label{eq:tilbetminusbet}
	\tilde\beta(n,r)= \beta(n,r) +\dfrac{n ( nr -r+1) b_2}{r},
\end{equation}
and thus if $b_2=0$, that is if $x_n$ is a polynomial in $n$ of degree at most one,  then  $\beta(n,r)=\tilde\beta(n,r)$.

\begin{prop}\label{betas}
	Let $r\ne 0 $,   let $\alpha(n,r)$ and $\beta(n,r)$ be the functions obtained with the vector of parameters $v=(b_0,b_1,b_2,s_1,s_2)$,  and let $\hat\alpha(n,r)$ and $\hat\beta(n,r)$ be the functions obtained with the vector $\hat v= (\hat b_0, \hat b_1, \hat b_2, \hat s_1, \hat s_2)$,   
	using equations \eqref{eq:alpharexpl} and  \eqref{eq:betapfd}. Let $\tilde\beta(n,r)$ be as defined in \eqref{eq:tilbetapfd}.

Define the following vectors of parameters 
\begin{eqnarray}\label{eq:vectors}
	w_1&=&(b_0+ s_1-\hat s_1,\  b_1,\  b_2,\  \hat s_1,\  s_2+ (\hat s_1-s_1) r ),\cr
	w_2&=&(-b_0 - s_1 -\hat s_1,\  -b_1,\  -b_2,\  \hat s_1,\  -s_2+(s_1+\hat s_1) r ),\cr
	w_3&=&\left( s_1 -\hat s_1 -\frac{s_2}{r},\ b_1 -  2 b_2  -\frac{b_2}{r},\ - b_2,\  \hat s_1,\ -r  (b_0 +s_1-\hat s_1)\right), \cr
 	w_4&=&\left(-s_1 -\hat s_1 + \frac{s_2}{r},\ - b_1  +2 b_2+\frac{b_2}{r},\  b_2,\   \hat s_1, \ r  ( b_0 +s_1+\hat s_1) \right). 
\end{eqnarray}
	 
	Let $T_k$ denote the substitution operator that sends $\hat v$ to $w_k$, for $1\le k \le 4$.
	Then   $ T_k \hat \alpha(n,r) = \alpha(n,r)$,  for $1\le k \le 4$, 
 and we have
	\begin{eqnarray}\label{eq:Tk}
		T_1(\hat\beta(n,r))&=&  \beta(n,r) +s_1-\hat s_1, \cr
		T_2(\hat\beta(n,r))&=&-(\beta(n,r) +s_1 +\hat s_1), \cr 
		T_3(\hat\beta(n,r))&=& \tilde\beta(n,r) +s_1 -\hat s_1, \cr
		T_4(\hat\beta(n,r))&=&-(\tilde\beta(n,r) +s_1+\hat s_1).
	\end{eqnarray}
\end{prop}

{\em Proof.} The proof is obtained by straightforward computations of the substitution operators $T_k$  acting on $\alpha(n,r)$ and $\beta(n,r)$.

Let us note that in the four equations in \eqref{eq:Tk} the right-hand side equals $-\hat s_1$ when $n=0$.

\begin{cor}\label{equalalfas}
If $\hat v=v$ then  $\hat\alpha(n,r)=\alpha(n,r)$ and $\hat\beta(n,r)=\beta(n,r)$ and the vectors $w_k$, for $1\le k \le 4$, become
\begin{eqnarray}\label{eq:simplevec}
	w_1&=&(b_0,\  b_1,\  b_2,\  s_1,\  s_2 ),\cr
	w_2&=&(-b_0 -2 s_1,\  -b_1,\  -b_2,\   s_1,\  -s_2+   2 r s_1 ),\cr
	w_3&=&\left( -\frac{s_2}{r},\ b_1 -  2 b_2  -\frac{b_2}{r},\ - b_2,\  s_1,\ -r  b_0 )\right), \cr
 	w_4&=&\left(-2 s_1  + \frac{s_2}{r},\ - b_1  +2 b_2+\frac{b_2}{r},\  b_2,\   s_1, \ r  ( b_0 + 2 s_1) \right), 
\end{eqnarray}
	$\alpha(n,r)$ is invariant under the substitutions $T_k$, for $1\le k \le 4$, and the equations \eqref{eq:Tk}
become
	\begin{eqnarray}\label{eq:Tksimple}
		T_1(\beta(n,r))&=&  \beta(n,r), \cr
		T_2(\beta(n,r))&=&-(\beta(n,r) + 2 s_1 ), \cr 
		T_3(\beta(n,r))&=& \tilde\beta(n,r), \cr
		T_4(\beta(n,r))&=&-(\tilde\beta(n,r) + 2 s_1).
	\end{eqnarray}
\end{cor}
	From this result we see that if $b_2$ is nonzero then for a given $\alpha(n,r)$  there correspond up to four different functions $\beta(n,r)$.

	\section{A classification of the hypergeometric orthogonal polynomial sequences}

	In this section we propose a new way to classify the hypergeometric orthogonal polynomial sequences.  We look at the different forms that the  coefficients of the three-term recurrence relation can have. In particular, we use the partial fractions decomposition of $\alpha(n,r)$ given in \eqref{eq:pfd}.

The parameter $r=a_2/a_1$ plays a central role in our construction of the classification.
If  $r$ is not in $\{0, 1, 1/2, 1/3\}$  then 
 the set  $\{y_0,y_1,y_2,y_3,y_4\}$ of functions of $n$ is clearly linearly independent. Therefore, for such values of $r$  we can classify the functions $\alpha(n,r)$ according to which of the 5 coefficients $d_k$ in the partial fractions decomposition are nonzero.

Notice that the coefficients $d_k$ are not independent.  Since $d_2=d_4-d_3$ it is clear that  if two of $d_2,d_3,d_4$ are zero so is the other one. Note also that when $\alpha(n,r)$ is a polynomial we have $d_3=d_4=0$ and hence $d_2=0$. In that case $\alpha(n,r)$ is a polynomial in $n$,  of degree 4 or degree 2, that has a root at $n=0$, and thus it coincides with an element of the set ${\mathcal C}_0$, that is, the same polynomial $\alpha(n,r)$ can be obtained by  taking $r=0$ and giving  appropriate values to the other parameters. 

For $r \in \C$ let $\cC_r$ be the set of all the functions $\alpha(n,r)$ defined in \eqref{eq:alphar} and let $\cC$ be the union of all the sets $\cC_r$.

 When $r$ is in $\{0,1,1/2, 1/3\}$ the set $\cC_r$  has some especial properties.  The set $\cC_0$ was described in section 3.  The other cases  will be considered later.

Let $r$ be a complex number that is not an element of  $\{0,1,1/2, 1/3\}$.  We construct a collection of subsets of $\cC_r$ as follows. To each $\alpha(n,r)$ in $\cC_r$ we associate the vector $(d_0,d_1, d_2, d_3, d_4)$ of the coefficients in the partial fractions decomposition  of $\alpha(n,r)$ given in \eqref{eq:pfd}. For each nonempty subset $A$ of $\{0,1,2,3,4\}$ we define the set $\cC_r(A)$ as the set of all the $\alpha(n,r)$ for which $d_k$ is nonzero if and only if $k \in A$. Some of the sets $\cC_r(A)$ turn out to be empty. We will see that there are only  19 sets $\cC_r(A)$ that are nonempty. From the definition it follows that if $A$ and $B$ are subsets of  $\{0,1,2,3,4\}$ such that $A \ne B$ then  $\cC_r(A)$ and  $\cC_r(B)$ are disjoint.

 If $A=\{j_1, j_2,\ldots,j_k\}$ is a subset of $\{0,1,2,3,4\}$, in order to simplify the notation we will write $\cC_r(j_1 j_2\cdots j_k)$ instead of $\cC_r(\{j_1, j_2,\ldots,j_k\})$, for example, if $A=\{1,3,4\}$ we write $\cC_r(134)$, with the elements of $A$ in increasing order.

Each set $\cC_r(A)$ is determined by some relations among the parameters $b_0, b_1, b_2, s_1, s_2$ that are obtained as follows.  We solve first the equations $d_k=0$, for $k$ not in $A$, for the values of $b_1,b_2, s_1, s_2$. This yields equations that express some of parameters in terms of the others and $r$ and $b_0$. Then we substitute such parameters in the polynomials $d_j$, for $j$ in $A$, and we obtain expressions that are required to  be nonzero and which are the coefficients in the partial fractions decomposition of the $\alpha(n,r)$ that belong to $\cC_r(A)$.

Since $d_2= d_4-d_3$ there are no solutions for sets of equations of the form  $d_k=0$ that contain exactly one of $d_2=0, d_3=0, d_4=0$. Each one of these 3 cases appears in 4 of the 31 possible sets of equations that correspond to the nonempty subsets of $\{0,1,2,3,4\}$. Therefore there are at most 19 nonempty classes $\cC_r(A)$. We will see that those 19 classes are indeed nonempty.

A direct computation yields
\begin{equation}\label{eq:alfa1r}
	\alpha(1,r)=\dfrac{(s_1+ b_0) ( s_2 - r s_1)}{r+1}, \qquad r\in \C, \quad r\ne -1.
\end{equation}
Since $\alpha(1,r)$ must be nonzero we must have $b_0 \ne -s_1$  and $ s_2 \ne r s_1$. 

If $r$ is not in $\{0,1,\frac{1}{2},\frac{1}{3}\}$ then the 19 classes are: 
 $$\cC_r(01234),$$  
 $$\cC_r(1234),\quad  \cC_r(0234),\quad  \cC_r(0134),\quad  \cC_r(0124),\quad  \cC_r(0123),$$
 $$\cC_r(234),\quad  \cC_r(134),\quad  \cC_r(124),\quad  \cC_r(123), \quad  \cC_r(034),\quad  \cC_r(024),\quad   \cC_r(023),$$
 $$\cC_r(34),\quad  \cC_r(24), \quad  \cC_r(23), \quad  \cC_r(01),$$
 $$ \cC_r(0), \quad  \cC_r(1).$$

Let us recall that all the $\alpha(n,r)$ in $\cC_r$ are invariant under the involution $R$.

Since $y_0(n,r)= 16 r^4 n^2 (rn-2 r +1)^2$  and $y_1(n,r)=4 r^3 n (rn-2r +1)$,
we see that the elements of $\cC_r(0)$ and $\cC_r(01)$  are polynomials in $n$ of degree 4 with a  zero at $n=0$ and  the elements of $\cC_r(1)$ are polynomials in $n$ of degree 2 with a zero at $n=0$. Therefore if $\alpha(n,r)$ is  an element of the union of $\cC_r(0), \cC_r(1)$ and $\cC_r(01)$  then it is  also an  element of the class $\cC_0$, that is $\alpha(n,r)$ can be obtained taking $r=0$ and giving appropriate values to the other parameters.

Let us note that the case $\alpha(n,r)=t$ for all $n\ge 1$, where $t$ is a constant,  is not in any of the 19 nonempty classes. Such case appears when $r$ is in $\{1, 1/2, 1/3\}$.

We present next some examples that  show how the classes $\cC_r(A)$ are characterized by means of some relations among the parameters $b_0,b_1,b_2, s_1, s_2$ and values that some of the parameters must have.

For the class $\cC_r(23)$ we first solve the polynomial system $ d_0=0, d_1=0, d_4=0$ with unknowns $ b_1,b_2,s_1,s_2$ and we obtain
$$ b_1=0, \qquad b_2=0, \qquad  s_2=2 r s_1 + r b_0$$
and $s_1, b_0$ are arbitrary numbers. Then we substitute these equations in $d_2$ and $d_3$  and obtain  
$$ d_2=64 r^6 (s_1 + b_0)^2, \qquad d_3=-64 r^6 (s_1 + b_0)^2.$$
Therefore we must have $s_1 + b_0\ne 0$, and by substitution in the general form of $\alpha(n,r)$ we obtain
\begin{equation}\label{eq:alf23}
\alpha(n,r)= \dfrac{(s_1+b_0)^2 r n ( rn -2r +1)}{(2rn -r+1)(2rn-3r+1)}, \qquad b_0 \ne -s_1.
\end{equation}
This is the general form of an element in $\cC_r(23)$.
In this case we get $\beta(n,r) = -s_1$, where $s_1$ is arbitrary.

The Gegenbauer polynomials \cite[eq. 9.8.22]{Hyp} are associated with this class. Their $\alpha(n,r)$ is obtained  with the  parameters  
$$r=\dfrac{1}{2 \lambda +1},\quad b_0=1,\quad b_1=0, \quad b_2=0, \quad  s_1=0, \quad s_2=r.$$

For the class $\cC_r(234)$ we solve the equations $ d_0=0$ and $d_1=0$ for $ b_1,b_2, s_1, s_2$ and we get $b_1=b_2=0$ and $s_1, s_2$ arbitrary. Substitution of these values in $d_2, d_3, d_4$ gives us 
\begin{eqnarray*}
d_2 &=& 16 r^4 ( b_0 r + s_2)^2, \cr
d_3 &=& -16 ( b_0 r^2 + ( b_0 + 2 s_1 +s_2) r -s_2)(b_0 r^2 +(-b_0 -2 s_1 +s_2)r +s_2), \cr
d_4 &=& 16 r^2 ( b_0 r +2 s_1 r -s_2)^2.
\end{eqnarray*}
 These equations impose some restrictions on $b_0, s_1, s_2$, since $d_2,d_3, d_4$ must be nonzero. 

 The general form of an element $\alpha(n,r)$ in $\cC_r(234)$ is
 \begin{equation}\label{eq:C234}
	 \alpha(n,r)= \dfrac{  -r n (rn -2r +1) w\, R(w) }{(2rn -r+1)(2rn-2r+1)^2(2rn-3r+1)},
 \end{equation}
where
$w= (rn-r+1)b_0 +(n-1)s_2 + s_1$ and $b_0, s_1, s_2$ must satisfy the restrictions mentioned above.

The $\alpha(n,r$) of the  Jacobi polynomials is in $\cC_r(234)$. It is obtained with 
$$r=\dfrac{1}{\alpha +\beta +2}, \quad b_0=1, \quad b_1=0, \quad b_2=0, \quad s_1=\dfrac{\alpha-\beta}{\alpha+\beta+2}, \quad s_2=\dfrac{1}{\alpha+\beta+2}, $$
where $\alpha$ and $\beta$ are the parameters used in \cite[eq. 9.8.5]{Hyp}.
The Pseudo-Jacobi polynomials are also associated with $\cC_r(234)$.

Now we consider the class $\cC_r(134)$. In this case we solve $d_0=0, d_2=0$ and we obtain $b_2=0$, $b_1$ and $s_2$ arbitrary,  and
$$s_1=\dfrac{(2-r^2) b_1^2+( 4 s_2-4 r b_0) b_1 + 4(r b_0+s_2)^2}{8 r b_1}.$$
Substitution in the general formulas for $d_1, d_3,d_4$ gives us
$d_1=-4 r^2 b_1^2$,  
$$d_3=r^2 b_1^{-2} (2 r b_0 + 2 s_2 + r b_1)^2(2 r b_0 + 2 s_2 - r b_1)^2,$$
and $d_4=d_3$. Therefore $b_1$, $2 r b_0 + 2 s_2 + r b_1$, and $2 r b_0 + 2 s_2 - r b_1$ must be nonzero.

In this case the general form of $\alpha(n,r)$ is
\begin{equation}\label{eq:alf134}
	\alpha(n,r)=\dfrac{-r n (rn - 2 r +1) w R(w)}{(2rn -r+1)(2rn -2r+1)^2 (2 rn - 3r +1)}
\end{equation}
where
$$w=(8 b_1 r)^{-1}(4 (r b_0+s_2) (r b_0 + s_2 +(2rn-2r+1) b_1)+((8 n^2 - 16 n + 7)r^2+(8n -8)r+2)b_1^2).$$

The class $\cC_r(123)$ is determined by $d_0=0, d_4=0$ and $d_1, d_2, d_3$ nonzero. These conditions give us $b_1=4r s_1+2 r b_0 -2 s_2, b_2=0$ and $b_0, s_1, s_2$  arbitrary, but such that $2 s_1 r + b_0 r -s_1\ne 0$, $ (2 s_1+b_0)r^2 -(s_2-b_0 )r - s_2 \ne 0$, and  $ (2 s_1+b_0)r^2 -(s_2+b_0 )r + s_2 \ne 0$. 
This class has an interesting property. The $\alpha(n,r)$ obtained when we take $s_1=0$ and the other parameters satisfying the conditions listed above is also in $\cC_r(123)$, it has the expression
\begin{equation}\label{eq:alfbeteq0}
	\alpha(n,r)=\dfrac{-n ( (n-2)r +1) ( (n-1) ( b_0 r -s_2) r - s_2) ( (n-1) ( b_0 r -s_2) + b_0)}{(2 r n -r +1)( 2 rn -3 n +1)},
\end{equation}
and it is the general form of all the $\alpha(n,r)$ for which the corresponding $\beta(n,r)=0$  for all $n \ge 0$. It is well-known that the Jacobi matrices with main diagonal equal to zero correspond to symmetric moment functionals. See \cite[Thm. 4.3]{Chi}.

\section{The classes with $r$ in  $\{1,\frac{1}{2},\frac{1}{3} \}$}

In this section we consider the sets $\cC_r(A)$ when $r=a_2/a_1$ is in $\{ 1, 1/2, 1/3\}$. Let us recall that we must have
$$\alpha(1,r)=\dfrac{(s_1+b_0) ( s_2- r s_1)}{r+1}\ne 0,$$
for every $r \ne -1$. Therefore $s_1+b_0$ and $ s_2- r s_1$ must be nonzero.

\subsection{The classes with $r=1$ }

Putting $r=1$  in \eqref{eq:alpharexpl} we obtain 
\begin{equation}\label{eq:alfn1}
	\alpha(n,1)=\dfrac{ (n ( x_{n-1} + s_2) + s_1 - s_2) ( (n-1) x_{-n} + s_2) - s_1 + s_2)}{4(2n-1)^2},
\end{equation}
which is the general form of $\alpha(n,1)$.
We also have
$$y_0(n,1)=16 n^2 (n-1)^2, \qquad y_1(n,1)=4 n (n-1), \qquad y_2(n,1)=1, $$
and
$$  y_3(n,1)=0, \qquad y_4(n,1)= \dfrac{-1}{(2n-1)^2}.$$

In this case we have a  class $\cC_1(A)$ for very nonempty subset of $\{0,1,2,4\}$. Let us note that $ y_0(1,1)=y_1(1,1)=0$ and thus the classes $\cC_1(0), \cC_1(1)$ and $\cC_1(01)$ are not acceptable because they yield $\alpha(1,1)=0$. The other 12 classes are nonempty and are acceptable.

The class $\cC_1(2)$ is determined by $b_1=0, b_2=0, s_2=2 s_1 +b_0$  and $s_1+b_0\ne 0$, where $s_1$ and $b_0$ are arbitrary. In this case we have
\begin{eqnarray}\label{eq:z2}
	\alpha(n,1)&=& \dfrac{(s_1+b_0)^2}{4}, \qquad n\ge 1,\cr
	\beta(n,1) &=& -s_1, \qquad n\ge 0.
\end{eqnarray}
The Chebyshev families are associated with this class.

For $\cC_1(124)$ we have $b_2=0$ and $b_1$, $4 b_1^2 + 16(s_2 - 2 s_1 -b_0)b_1+16(s_2 +b_0)^2$, and $4 s_1+2 b_0 - 2 s_2 -b_1$ must be nonzero. The general $\alpha(n,1)$ in this class is
\begin{equation}\label{eq:z124}
	\alpha(n,1)=\dfrac{-(n^2 b_1+(s_2+b_0 -b_1)n + s_1-s_2)(n^2 b_1-(s_2+b_0+ b_1)n + s_1 + b_0)}{4(2 n-1)^2}.
\end{equation}

For $\cC_1(24)$ we have $b_1=b_2=0$, $s_2+b_0 \ne 0$, and $ 2 s_1+b_0-s_2 \ne 0$. Its elements are expressed as
\begin{equation}\label{eq:z24}
	\alpha(n,1)=\dfrac{((s_2+b_0)n + s_1 - s_2)((s_2+b_0)n -s_1-b_0)}{4 (2n-1)^2}.
\end{equation}
Note that the numerator has degree 2 and does not vanish at $n=0$.
There are other six classes whose elements have $(2n-1)^2$ as a factor of the denominator.  

The elements of $\cC_1(02), \cC_1(12)$ and $\cC_1(012)$ are polynomials of degree at most 4 and have nonzero constant term  because $y_2=1$. Therefore those classes are not contained in $\cC_0$.

For $r=1$ the classes are 
$$ \cC_1(0124),$$
$$\cC_1(124),\quad  \cC_1(024),\quad \cC_1(014),\quad \cC_1(012),   $$
$$\cC_1(24), \quad \cC_1(14),\quad \cC_1(04), \quad \cC_1(12),\quad  \cC_1(02),  $$
$$ \cC_1(2),\quad  \cC_1(4). $$

\subsection{The classes with $r=\frac{1}{3}$}

When $r=1/3$ the involution $R$ becomes $R(u(n))=u(-n-1)$ for any rational function $u(n)$. Therefore  the elements of $\cC_{1/3}$ satisfy $\alpha(n,1/3)=\alpha(-n-1,1/3)$. 

From \eqref{eq:alpharexpl} we obtain 
\begin{equation}\label{eq:alfn13}
	\alpha(n,1/3)=\dfrac{((n+2) x_{n-1} +3((n-1)s_2+ s_1)) ((n-1) x_{-n-2} + 3 ((n+2) s_2-s_1))}{4 (2n+1)^2}.
	\end{equation}
This is the general form of the elements in $\cC_{1/3}$.

In this case we have
$$y_0(n,1/3)=\dfrac{16}{729} n^2 (n+1)^2, \qquad y_1(n,1/3)=\dfrac{4}{81} n (n+1), \qquad y_2(n,1/3)=1, $$
and
$$  y_3(n,1/3)=0, \qquad y_4(n,1/3)= \dfrac{-1}{(2n+1)^2}.$$
Therefore  there is a class $\cC_{1/3}(A)$ for every nonempty subset of $\{0,1,2,4\}$. 

The elements of $\cC_{1/3}(0), \cC_{1/3}(1)$ and $\cC_{1/3}(01)$ are polynomials with constant term equal to zero and thus we can consider that these classes are contained in $\cC_0$. 

The class $\cC_{1/3}(2)$ is determined by $b_1=b_2=0$, $s_1=(3 s_2-b_0)/2$ and $3 s_2 + b_0 \ne 0$. These conditions give us
\begin{equation*}
	\alpha(n,1/3)=\dfrac{(3 s_2 + b_0)^2}{16}, \qquad n\ge 1.
\end{equation*}
This is similar to \eqref{eq:z2}. Therefore $\cC_{1/3}(2)$ is equivalent to $\cC_1(2)$.

For $\cC_{1/3}(4)$ we have $b_1=b_2=0, s_2=-b_0/3$ and $s_1+b_0\ne 0$, and then the elements in this class are of the form
\begin{equation}\label{eq:st4}
	\alpha(n,1/3)=\dfrac{-9 (s_1+b_0)^2}{4(2n +1)^2}.
\end{equation}
There are other 6 classes whose elements have as denominator a constant multiple of $(2n+1)^2$.

For $\cC_{1/3}(12)$ we have $b_1=(4 s_1 + 2 b_0)/3$, $ b_2=0$,   $2 s_1+b_0 -3 s_2\ne 0$, $ s_1+2 b_0 + 3 s_2 \ne 0$, and $s_1 -b_0 -6 s_2 \ne 0$. 
The general element in this class is of the form
\begin{equation}\label{eq:st12}
	\alpha(n,1/3)=\dfrac{-1}{36} ((2 s_1 - 3 s_2 +b_0)n + s_1 + 3 s_2 + 2 b_0) ((2 s_1 - 3 s_2 + b_0)n + s_1 - 6 s_2 - b_0).
\end{equation}
This is a polynomial of degree 2 with nonzero constant term. The elements of $\cC_{1/3}(02)$ and $\cC_{1/3}(012)$ are also polynomials with nonzero constant term.

For $r=1/3$ the classes are 
$$ \cC_{1/3}(0124),$$
$$\cC_{1/3}(124),\quad  \cC_{1/3}(024),\quad \cC_{1/3}(014),\quad \cC_{1/3}(012),   $$
$$\cC_{1/3}(24),\quad \cC_{1/3}(14),\quad \cC_{1/3}(04),\quad \cC_{1/3}(12),\quad \cC_{1/3}(02),  $$
$$ \cC_{1/3}(2),\quad \cC_{1/3}(4). $$
In this list we have deleted the 3  classes that are included in $\cC_0$. 

\subsection{The classes with $r=\frac{1}{2}$}

If we put $r=1/2$ in \eqref{eq:alpharexpl} we obtain
\begin{equation}\label{eq:alfn12}
	\alpha(n,1/2)=\dfrac{((n+1) x_{n-1} +2((n-1)s_2+ s_1)) ((n-1) x_{-n-1} + 2 ((n+1) s_2-s_1))}{4 (4n^2-1)}.
	\end{equation}
This is the general form of the elements in $\cC_{1/2}$.

In this case we have
$$y_0(n,1/2)=\dfrac{n^4}{4}, \qquad y_1(n,1/2)=\dfrac{n^2}{4}, \qquad y_2(n,1/2)=1,\qquad y_3(n,1/2)=\dfrac{-1}{4n^2-1}, $$
and $y_4(n,1/2)=0$.

Let us note that $y_0(n,1/2)$ and $y_1(n,1/2)$ are polynomials with constant term equal to zero. Therefore the elements of $\cC_{1/2}(0)$, $\cC_{1/2}(1)$, and $\cC_{1/2}(01)$ are also polynomials with constant term equal to zero and thus these classes can be considered as subsets of $\cC_0$.

The set $\cC_{1/2}(2)$ is determined by the conditions $b_1=b_2=0$, $s_1=s_2/2 - 3 b_0 /4$, and $2 s_2 + b_0 \ne 0$. In this case we have
$$\alpha(n,1/2)=\dfrac{(2 s_2+b_0)^2}{16}.$$
This is similar to the elements of $\cC_1(2)$  and $\cC_{1/3}(2)$.

For $\cC_{1/2}(12)$ we have $b_2=0$, $s_1=s_2/2+3 b_1/8-3 b_0/4$, $b_1\ne0$,  and $s_2 +b_0/2+ b_1/4 \ne 0$.
The elements in this class are of the form
\begin{equation*}
	\alpha(n,1/2)= \dfrac{1}{64} ( 2 b_1 n +4 s_2 +2 b_0 + b_1)(-2 b_1 n+ 4 s_2 + 2 b_0 + b_1).
\end{equation*}
This is a polynomial of degree 2 with nonzero constant term. The elements of the classes $\cC_{1/2}(02)$ and $\cC_{1/2}(012)$ are polynomials of degree 4 with nonzero constant term.

The elements of $\cC_{1/2}(3)$  have $ b_1=b_2=0$, $s_2= -b_0/2$, and $ s_1 + b_0 \ne 0$ and are of the form
\begin{equation}\label{eq:sh3}
	\alpha(n,1/2)= \dfrac{-( s_1+ b_0)^2}{4 n^2-1}.
\end{equation}
The elements of all the classes corresponding to subsets of $\{0,1,2,3\}$ that contain 3 also have a scalar multiple of $4 n^2 -1$ as denominator.

For $r=1/2$ the classes are 
$$ \cC_{1/2}(0123),$$
$$\cC_{1/2}(123),\quad  \cC_{1/2}(023),\quad \cC_{1/2}(013),\quad \cC_{1/2}(012),   $$
$$\cC_{1/2}(23),\quad \cC_{1/2}(13),\quad \cC_{1/2}(03),\quad \cC_{1/2}(12),\quad \cC_{1/2}(02),  $$
$$ \cC_{1/2}(2),\quad \cC_{1/2}(3). $$

{\bf Remark.} The classes $\cC_r(02),\cC_r(012), \cC_r(12)$, for $r\in \{1, 1/3,1/2\}$ contain polynomials of degree exactly equal to  4 or  2,   with nonzero constant term, but the intersection of corresponding classes with different $r$ is empty because $\alpha(n,1)= \alpha(-n+1,1)$, $\alpha(n,1/3)=\alpha(-n-1,1/3)$ and $\alpha(n,1/2)= \alpha(-n,1/2)$ for all $n$. Only constant polynomials satisfy two of these conditions.

\section{The polynomial sequences in the Askey scheme}
In this section for each of the polynomial sequences in the Askey scheme we indicate with which class it  is associated.

In the book \cite{Hyp} the recurrence coefficients that we call $\alpha_n$ are written as rational functions of $n$ with the  numerator and written as a product of terms of the form $ n+ t_j$, where the $ t_j$ are constants or parameters.
From that form it is easy to find the values of the parameter $r$, the other parameters, and the coefficients $d_k$ in the partial fractions decomposition of $\alpha(n,r)$. This is done  by solving some simple equations.

There are 18 families of orthogonal polynomial sequences listed in \cite{Hyp}.
For each one of them we identify its corresponding  class $\cC_r(A)$. 

The recurrence coefficients of the Wilson polynomials are defined in \cite[eq. 9.1.5]{Hyp} in terms of the  parameters $a,b,c,d$. In this case our parameters are 
$$ r=\dfrac{1}{a+b+c+d},\quad b_0=-a^2, \quad b_1=-2 a-1, \quad b_2=-1, $$
$$ s_1= -r ((b+d) a + d b)c + abc), \qquad s_2= -r(bc+ d b + cd).$$
There are other expressions for these parameters because $a,b,c,d$ appear in a  symmetric way in  $\alpha(n,r)$. 
If the parameters $a,b,c,d$ satisfy certain conditions then the corresponding $\alpha(n,r)$ is in $\cC_r(01234)$, where $r=(a+b+c+d)^{-1}$, that is all the $d_k$ are nonzero. 
Note that $b_2^2=d_0=1$,  independently of the values of $a,b,c,d$. Therefore giving values to the parameters in the $\alpha(n,r)$ of the Wilson polynomials it is not possible to obtain elements in $\cC_r(A)$ if $A$ does not contain 0. 

The recurrence coefficients of the Racah polynomials are defined in \cite[eq. 9.24]{Hyp} using the parameters $\alpha, \beta, \gamma, \delta, N $. In the case with $\alpha+1=-N$  we obtain
$$ r=\dfrac{-1}{N-\beta -1}, \qquad b_0=0, \qquad b_1= \delta + \gamma+2, \qquad b_2=1,$$
$$  s_1= -r (\gamma+1) ( \beta + \delta +1) N, \qquad s_2= -r ( ( \delta+N) \beta - \delta \gamma + N).$$ 
The other cases are similar. If the parameters satisfy certain conditions we can have $d_k\ne 0$, for $0 \le k \le 4.$
Note that we have $b_2^2=d_0=1$, as in the case of the Wilson polynomials.

For the continuous Hahn polynomials the recurrence coefficients are defined in \cite[eq, 9.4.4]{Hyp} with parameters $a,b,c,d$. In this case we obtain
$$ r=\dfrac{1}{a+b+c+d}, \quad b_0=i a, \quad b_1=i, \quad b_2=0,
\quad s_1=i r ( cd-ab), \quad s_2=-i r b,$$
and in the general case the $\alpha(n,r)$ of the continuous Hahn polynomials is in $\cC_r(1234)$, with the $r$ defined above. Let us note that $d_1=4 r^2$ is never equal to zero. 

The recurrence coefficients of the Jacobi polynomials, with parameters $\alpha$ and $\beta$, are defined in \cite[eq. 9.8.5]{Hyp}. 
In this case we have
$$ r=\dfrac{1}{2 +\alpha+\beta}, \quad b_0=1, \quad  b_1=0,\quad  b_2=0,\quad s_1=r(\alpha-\beta), \quad s_2= r.$$
These values give us  the coefficients $d_0=0, d_1=0$ and
$$d_2=\dfrac{64}{(2+\alpha+\beta)^6}, \qquad d_3=\dfrac{64( (\alpha - \beta)^2-1)}{(2+\alpha+\beta)^6}, \qquad d_4=\dfrac{64 (\alpha - \beta)^2}{(2+\alpha+\beta)^6},$$
of the partial fractions decomposition. Note that $d_2$ is nonzero for all values of $\alpha$ and $\beta$. This means that the general Jacobi polynomials are associated with $\cC_r(234)$, where $r=(2 +\alpha+\beta)^{-1}$. 

For the Gegenbauer polynomials, with a parameter $\lambda$, the recurrence coefficients are defined in \cite[eq. 9.8.22]{Hyp}. For this family we have
$$r=\dfrac{1}{2 \lambda+1}, \qquad b_0=1, \qquad b_1=0, \qquad b_2=0,\qquad s_1=0, \qquad s_2=r.$$
The coefficients in the partial fraction decomposition are 
$$d_0=0, \qquad d_1=0,\qquad d_2=\dfrac{64}{(2 \lambda+1)^6},\qquad  d_3=\dfrac{64}{(2 \lambda+1)^6}, \qquad d_4=0.$$
Therefore the Gegenbauer polynomials are associated with the class $\cC_r(23)$, where $r=(2 \lambda+1)^{-1}$, and for any $\lambda$ we see that  $d_2$ and $d_3$ are nonzero.

The polynomials in the Chebyshev family \cite[eq:9.8.41]{Hyp} have $\alpha_n=1/4$ for all $n\ge 1$. The sequences $\alpha_n$ equal to a  constant are in $\cC_{1/2}(2)$, $\cC_1(2)$, and $\cC_{1/3}(2)$. The parameters for these cases are the following. For $r=1/2$
$$r=1/2, \qquad b_0=1, \qquad  b_1=0, \qquad b_2=0,\qquad  s_1=1/2, \qquad s_2=1/2.$$
For $r=1$ 
$$r=1, \qquad b_0=1, \qquad  b_1=0, \qquad b_2=0,\qquad  s_1=0, \qquad s_2=1.$$
For $r=1/3$
$$r=1/3, \qquad b_0=1, \qquad  b_1=0, \qquad b_2=0,\qquad  s_1=0, \qquad s_2=1/3.$$

If $r=1/2$ we obtain $\beta_0=-1/2$ and $\beta_n=0$ for $n\ge 1.$ In the other two cases we have $\beta_n=0$ for all $n\ge 0$.

The Legendre polynomials \cite[eq. 9.8.65]{Hyp} have $\alpha_n=n^2/(4 n^2-1)$. 
For this family we have
$$ r=1/2,\qquad b_0=1, \qquad b_1=0, \qquad b_2=0, \qquad s_1=0, \qquad s_2=1/2. $$
The Legendre polynomials correspond to the class $\cC_{1/2}(23)$. 

For the Bessel polynomials \cite[eq. 9.13.4]{Hyp} with parameter $a$ 
we have 
$$r=\dfrac{1}{a+2}, \qquad b_0=0, \qquad b_1=0, \qquad b_2=0, \qquad s_1=\dfrac{2}{a+2}, \qquad s_2=0.$$
The Bessel polynomials correspond to the class $\cC_r(34)$, where $r=1/(a+2)$. 

The pseudo-Jacobi polynomials \cite[eq. 9.9.4]{Hyp} are associated with the class $\cC_r(234)$, where $r=-1/(2N)$. This case is similar to the one of Jacobi polynomials, but with different parameters.

The Hahn polynomials \cite[eq. 9.5.4]{Hyp}, with parameters $\alpha, \beta, N$, are associated with the class $\cC_r(1234)$, where $r=1/(\alpha+\beta+2)$. This is similar to the case of the continuous Hahn polynomials.

The rest of the polynomials in the Askey scheme not mentioned above are associated with the class $\cC_0$, that is, their sequences $\alpha_n$ are polynomials in $n$ of degree at most 4 that have $\alpha_0=0$. This is the case of the continuous dual Hahn,  dual Hahn, Meixner-Pollaczek, Meixner, 
 Krawtchouk, Laguerre, Charlier, and Hermite polynomials. Therefore the elements of the Askey scheme are associated with only six classes $\cC_r(A)$, with several different values of $r$ and some subsets $A$ of $\{0,1,2,3,4\}$.

In the literature we can find  a large number of hypergeometric orthogonal polynomial sequences that are not listed in the Askey scheme. Most of them are obtained by modifications of elements in the Askey scheme, giving certain values to the parameters, taking limits of some parameters, or modifying the measure or moments functional that determines the orthogonality.

Our classification allows us to construct elements in each class in a systematic algebraic way, without using limits nor modification of measures.

\end{document}